\documentclass[a4paper]{amsart}
\usepackage[utf8]{inputenc}
\usepackage[T1]{fontenc}
\usepackage{lmodern}
\usepackage[all]{xy}
\usepackage{nicefrac,mathtools,enumitem}
\usepackage{microtype}
\usepackage{ifthen}
\usepackage{leftidx}
\usepackage{amssymb}
\usepackage{amsxtra}

\usepackage{tikz}
\usetikzlibrary{matrix}
\tikzset{cd/.style=matrix of math nodes,row sep=2em,column sep=2em, text height=1.5ex, text depth=0.5ex}
\tikzset{cdar/.style=->,auto}

\setlist[enumerate,1]{label=\textup{(\arabic*)}}
\setlist[enumerate,2]{label=\textup{(\alph*)}}

\usepackage[pdftitle={The Haagerup property for Drinfeld doubles},
pdfauthor={Sutanu Roy},
pdfsubject={Mathematics}
]{hyperref}
\usepackage[lite]{amsrefs}
\newcommand*{\MRref}[2]{ \href{http://www.ams.org/mathscinet-getitem?mr=#1}{MR #1}}
\newcommand*{\arxiv}[1]{ \href{http://www.arxiv.org/abs/#1}{arXiv:#1}}
\renewcommand{\PrintDOI}[1]{\href{http://dx.doi.org/#1}{DOI #1}%
  \IfEmptyBibField{volume}{, (to appear in print)}{}}

\renewcommand{\PrintDOI}[1]{\href{http://dx.doi.org/\detokenize{#1}}{doi: \detokenize{#1}}%
  \IfEmptyBibField{pages}{, (to appear in print)}{}}

\BibSpec{book}{%
    +{}  {\PrintPrimary}                {transition}
    +{,} { \textit}                     {title}
    +{.} { }                            {part}
    +{:} { \textit}                     {subtitle}
    +{,} { \PrintEdition}               {edition}
    +{}  { \PrintEditorsB}              {editor}
    +{,} { \PrintTranslatorsC}          {translator}
    +{,} { \PrintContributions}         {contribution}
    +{,} { }                            {series}
    +{,} { \voltext}                    {volume}
    +{,} { }                            {publisher}
    +{,} { }                            {organization}
    +{,} { }                            {address}
    +{,} { \PrintDateB}                 {date}
    +{,} { }                            {status}
    +{}  { \parenthesize}               {language}
    +{}  { \PrintTranslation}           {translation}
    +{;} { \PrintReprint}               {reprint}
    +{.} { }                            {note}
    +{.} {}                             {transition}
    +{,} { \PrintDOI}                   {doi}
    +{}  {\SentenceSpace \PrintReviews} {review}
}

\numberwithin{equation}{section}

\theoremstyle{plain}
\newtheorem{theorem}[equation]{Theorem}
\newtheorem{result}[equation]{Result}

\newtheorem{proposition}[equation]{Proposition}

\newtheorem{corollary}[equation]{Corollary}

\theoremstyle{definition}
\newtheorem{definition}[equation]{Definition}

\theoremstyle{remark}
\newtheorem{remark}[equation]{Remark}

\newtheorem{example}[equation]{Example}

\newcommand*{\nb}{\nobreakdash}
\newcommand*{\Star}{\(^*\)\nb-}

\newcommand*{\C}{\mathbb C}

\newcommand*{\Bialg}[1]{(#1,\Comult[#1])}%C*-bialgebra
\newcommand*{\DuBialg}[1]{(\hat{#1},\DuComult[#1])}%dual C*-bialgebra

\newcommand*{\G}[1][G]{\mathbb #1}% quantum group 
\newcommand*{\DuG}[1][G]{\widehat{\mathbb{#1}}}%dual quantum group

\newcommand*{\Qgrp}[2]{\mathbb{#1}=\Bialg{#2}}%quantum group as pair
\newcommand*{\DuQgrp}[2]{\widehat{\mathbb{#1}}=\DuBialg{#2}}%dual quantum group as pair

\newcommand*{\Comult}[1][]{\Delta_{#1}}%comultiplication
\newcommand*{\DuComult}[1][]{\hat{\Delta}_{#1}}%dualcomultiplication

\newcommand*{\Bound}{\mathbb B}%bounded operators on a Hilbert space
\newcommand*{\Comp}{\mathbb K}%compact operators on a Hilbert module
\newcommand*{\Contvin}{\textup C_0}%continuous functions vanishing at infinity
\newcommand*{\Mor}{\textup{Mor}}%nondegenerate *-homomorphisms of C*-algebras
\newcommand*{\Id}{\textup{id}}%identity map

\newcommand*{\Multunit}[1][]{\mathbb{W}^{#1}}%muliplicative unitary acting on a hilbert space
\newcommand*{\multunit}[1][]{\textup{W}^{#1}}%multiplicative unitary as a bicharacter of the multiplier algebra
\newcommand*{\DuMultunit}[1][]{\widehat{\mathbb{W}}{}^{#1}}%muliplicative unitary on action on a hilbert space
\newcommand*{\bichar}{\textup{V}}%bicharacter viewed as an element of of the unitary multiplier

\newcommand*{\corep}[1]{\textup{#1}}                %Corepresentation of a quantum group
\newcommand*{\Corep}[1]{\mathbb{#1}}             %Corepresentation as operator on Hilbert space
\newcommand*{\Drinfdouble}[1]{\mathfrak{D}(#1)}%Drinfeld double
\newcommand*{\Codouble}[1]{\mathfrak{D}({#1})\sphat\text{\space}}% Dual of Drinfeld double 

\newcommand*{\DrinfDoubAlg}{\mathcal{D}}%Drinfeld double algebra
\newcommand*{\CodoubAlg}{\widehat{\mathcal{D}}}%Quantum codouble algebra

\newcommand*{\Flip}{\Sigma}% flip operator on Hilbert space
\newcommand*{\flip}{\sigma}% flip map on the multiplier algebra
\newcommand*{\Cst}{\textup C^*}%C*-algebra
\newcommand*{\Cred}{\textup C^*_\textup r}%reduced group C*-algebra
\newcommand*{\CLS}{\mathrm{CLS}}%closed linear span
\newcommand*{\Cstcat}{\mathfrak{C^*alg}}%category of C*-algebras
\newcommand*{\Hils}[1][H]{\mathcal #1}%Hilbert space
\newcommand*{\Mult}{\mathcal M}%multiplier algebra
\newcommand*{\U}{\mathcal U}%unitary group
\newcommand*{\defeq}{\mathrel{\vcentcolon=}}

\newcommand*{\norm}[1]{\lVert#1\rVert}

\begin{document}
\title{The Haagerup property for Drinfeld doubles}

\author{Sutanu Roy}
\email{sr26@uOttawa.ca}
\address{Department of Mathematics and Statistics\\
              585 King Edward\\ 
              K1N 6N5 Ottawa\\
              Canada}
  
\begin{abstract}
 We show that Drinfeld's 
 double group construction for locally compact quantum groups  
 preserves the Haagerup property. This shows that the   
 Drinfeld doubles of the quantum groups, \(\Contvin(\mathbb{F}_{2})\), 
 \(\textup{SU}_{q}(2)\), \(\textup{SU}_{q}(1,1)_{\text{ext}}\), 
 quantum \(ax+b\), quantum~\(az+b\), and 
 \(\textup{E}_{q}(2)\) have the Haagerup property.
\end{abstract}

\maketitle

\section{Introduction}
  \label{sec:intro} 
  The Haagerup approximation property for groups 
  is weaker than the notion of amenability. 
  One of the equivalent formulations 
  of the Haagerup property for groups 
  is the existence of a mixing representation 
  that weakly contains the trivial representation. 
  Following this characterisation,   
  Daws, Fima, Skalski and White introduced 
  the Haagerup property for locally compact quantum groups 
  in~\cite{Daws-Fima-Skalski-White:Haagerup_prop_qnt_grp}.
  Moreover, \cite{Daws-Fima-Skalski-White:Haagerup_prop_qnt_grp}*{Proposition 5.2} 
  shows that the Haagerup property for a quantum group~\(\G\) 
  follows from the coamenability of the dual quantum group 
  \(\DuG\). As an application, they show that 
  quantum~\(ax+b\) in~\cite{Woronowicz:Quantum_azb}, 
  quantum~\(az+b\) in~\cite{Woronowicz-Zakrzewski:Quantum_axb},
  and~\(\textup{E}_{q}(2)\) in~\cite{Woronowicz:Qnt_E2_and_Pontr_dual}, 
  and their duals have the Haagerup property 
  (see \cite{Daws-Fima-Skalski-White:Haagerup_prop_qnt_grp}*{Example 5.4}).
  All of these examples are amenable and coamenable. 
  Also, like compact groups, all compact quantum groups 
  have the Haagerup property.
  The quantum group~\(\Contvin(\mathbb{F}_{2})\) associated to 
  the free group~\(\mathbb{F}_{2}\) is an example of a discrete 
  non\nb-amenable quantum group with the Haagerup property.
  The recent work of Caspers~\cite{Caspers:Weak_amenb_SU_1_1} 
  shows that the extended~\(\textup{SU}_{q}(1,1)\) groups 
  in \cites{Koelink-Kustermans:normalizer_SU_1_1}, 
  are examples of non\nb-classical, non\nb-discrete and 
  non\nb-amenable quantum groups that enjoy the Haagerup property. 
  Moreover, their duals also have the 
  Haagerup property. 
  
  The quantum double construction of Drinfeld 
  or the \emph{Drinfeld double} for Hopf algebras 
  is one of the fundamental results of the pioneering 
  work of Drinfeld \cite{Drinfeld:Quantum_groups}.
  Roughly, the Drinfeld double of a finite dimensional 
  Hopf algebra~\(H\) over a field~\(k\) 
  is a Hopf algebra~\(\Drinfdouble{H}\) such that the factors 
  \(H\) and \(H^{*}\defeq\textup{Hom}_{k}(H,k)\) inside
  \(\Drinfdouble{H}\) do not commute. The quantum double 
  construction for analytic quantum groups was developed 
  in many different frameworks, along with the development 
  of a general theory of compact and locally compact quantum 
  groups. In~\cite{Podles-Woronowicz:Quantum_deform_Lorentz}, 
  Podle\'s and Woronowicz introduced the \emph{double group 
  construction}, as the dual of the Drinfeld double, 
  for compact quantum groups~\cite{Woronowicz:CQG}. 
  In~\cite{Baaj-Vaes:Double_cros_prod}, Baaj and Vaes 
  obtained Drinfeld double for regular 
  (in the sense of Baaj and Skandalis~\cite{Baaj-Skandalis:Unitaires}) 
  \(\Cst\)\nb-algebraic locally compact quantum groups 
  (see \cites{Kustermans-Vaes:LCQG, Masuda-Nakagami-Woronowicz:C_star_alg_qgrp}) 
  as a special case of the dual of the generalised double crossed product construction. 
  
  In general, for \(\Cst\)\nb-algebraic locally compact quantum groups 
  (see \cites{Kustermans-Vaes:LCQG, Masuda-Nakagami-Woronowicz:C_star_alg_qgrp}), 
  this is generalised by Masuda, Nakagami and Woronowicz 
  \cite{Masuda-Nakagami-Woronowicz:C_star_alg_qgrp}*{Section 8}, under the name
  \emph{quantum codouble}. We shall follow this terminology. 
  (What we call quantum codouble is called Drinfeld double 
  in~\cite{Nest-Voigt:Poincare}*{Section 3}).
   
  The main result of this article is the following theorem:
   \begin{result}
    \label{res:Haagerup_prop_drinf_doub}
    Drinfeld's double group construction preserves the Haagerup property. 
    That is, the Drinfeld double of~\(\G\) has the Haagerup property 
    whenever the quantum group~\(\G\) and its dual, \(\DuG\) both have 
    the Haagerup property. 
  \end{result}   
  
   After introducing basic definitions in Section~\ref{sec:LCQG},
   we briefly extract the~\(\Cst\)\nb-bialgebra structure 
   from their construction in Section~\ref{sec:drinf}.
   In the last Section~\ref{sec:Haagerup_prop}, we prove our  
   main result~\ref{res:Haagerup_prop_drinf_doub}. Then, 
   we show that the Drinfeld doubles of the classical group~\(\mathbb{F}_{2}\), 
   quantum~\(ax+b\), quantum~\(az+b\), quantum~\(\textup{E}(2)\), 
   quantum~\(\textup{SU}(2)\) and extended quantum~\(\textup{SU}(1,1)\)
   have the Haagerup property.  
   
   \subsection{Basic notation}
   All Hilbert spaces and \(\Cst\)\nb-algebras are assumed to be separable. 
   For two norm\nb-closed subsets~\(X\) and~\(Y\) of a~\(\Cst\)\nb-algebra, 
   let 
   \[
     X\cdot Y\defeq\{xy : x\in X, y\in Y\}^{\textup{CLS}},
   \]
  where CLS stands for the~\emph{closed linear span}.
  
  For a~\(\Cst\)\nb-algebra~\(A\), let~\(\Mult(A)\) be its multiplier algebra and 
   let \(\U(A)\) be the group of unitary multipliers of~\(A\).  
  Let~\(\Cstcat\) be the category of \(\Cst\)\nb-algebras with
  nondegenerate \Star{}homomorphisms \(\varphi\colon A\to\Mult(B)\) as
  morphisms \(A\to B\); let \(\Mor(A,B)\) denote this set of morphisms.
  
   For a Hilbert space~\(\Hils\), \(\Comp(\Hils)\) and~\(\Bound(\Hils)\) 
   denote the \(\Cst\)\nb-algebras of compact and bounded operators 
   acting on~\(\Hils\), respectively. A \emph{representation} of a
  \(\Cst\)\nb-algebra~\(A\) on a Hilbert space 
  is an element of~\(\Mor(A,\Comp(\Hils))\). The group of unitary 
  operators on a Hilbert space~\(\Hils\) 
  is denoted by \(\U(\Hils)\).

  We write~\(\Flip\) for the tensor flip \(\Hils\otimes\Hils[K]\to
  \Hils[K]\otimes\Hils\), \(x\otimes y\mapsto y\otimes x\), for two 
  Hilbert spaces \(\Hils\) and~\(\Hils[K]\).  We write~\(\flip\) for the
  tensor flip isomorphism \(A\otimes B\to B\otimes A\) for two
  \(\Cst\)\nb-algebras \(A\) and~\(B\), where~\(\otimes\) denotes 
  the minimal tensor product of~\(\Cst\)\nb-algebras. 
   
 \section{Locally compact quantum groups}
  \label{sec:LCQG}
   For a general theory of \(\Cst\)\nb-algebraic locally compact quantum groups 
   see~\cites{Masuda-Nakagami-Woronowicz:C_star_alg_qgrp, 
   Kustermans-Vaes:LCQG}. 
      
   \begin{definition}[\cite{Baaj-Skandalis:Unitaires}*{D\'efinition 0.1}]
    \label{def:C_star_bialg} 
    A~\(\Cst\)\nb-\emph{bialgebra}~\(\Bialg{A}\) is a~\(\Cst\)\nb-algebra 
    \(A\) and a comultiplication~\(\Comult[A]\in\Mor(A,A\otimes A)\)
    that is coassociative:
    \((\Id_{A}\otimes\Comult[A])\circ\Comult[A]=(\Comult[A]\otimes\Id_{A})\circ\Comult[A]\).
    Moreover, if~\(\Comult[A]\) satisfies the cancellation property, 
    \[
      \Comult[A](A)\cdot (1_{A}\otimes A)=\Comult[A](A)\cdot (A\otimes 1_{A})=A\otimes A,
    \]
    \(\Bialg{A}\) is a~\emph{bisimplifiable} \(\Cst\)\nb-bialgebra. 
   \end{definition} 
   Let~\(\varphi\) be a faithful (approximate) KMS weight
   (see~\cite{Kustermans-Vaes:LCQG}*{Section 1}) 
   on~\(A\). The set of all positive 
   \(\varphi\)\nb-integrable elements is 
   defined by~\(\mathcal{M}_{\varphi}^{+}
   \defeq\{a\in A^{+}:\varphi(a)<\infty\}\). 
   Moreover, \(\varphi\) is  called
   \begin{enumerate} 
   \item \emph{left invariant} if     
   \(\omega((\Id_{A}\otimes\varphi)\Comult[A](a))=\omega(1)\varphi(a)\)
     for all~\(\omega\in A_{*}^{+}\), \(a\in\mathcal{M}^{+}_{\varphi}\);
   \item \emph{right invariant} if       
   \(\omega((\varphi\otimes\Id_{A})\Comult[A](a))=\omega(1)\varphi(a)\)
   for all~\(\omega\in A_{*}^{+}\), \(a\in\mathcal{M}^{+}_{\varphi}\).
   \end{enumerate}
   \begin{definition}[\cite{Kustermans-Vaes:LCQG}*{Definition 4.1}]
    \label{def:Qnt_grp}
    A \emph{locally compact quantum group} 
    (\emph{quantum groups} from now onwards) is 
    a bisimplifiable~\(\Cst\)\nb-bialgebra~\(\Qgrp{G}{A}\) with left
    and right invariant approximate KMS weights~\(\varphi\) and~\(\psi\), 
    respectively. 
   \end{definition}
   By Theorem~\(7.14\) and~\(7.15\) 
   in~\cite{Kustermans-Vaes:LCQG}, the invariant 
   weights~\(\varphi\) and~\(\psi\) are unique up to a 
   positive scalar factor; 
   hence they are called the left and right 
   \emph{Haar weights} for \(\G\). Moreover,  
   there is a unique (up to isomorphism) Pontrjagin 
   dual~\(\DuQgrp{G}{A}\) of~\(\G\), which is again a quantum group. 
   
   Next we consider the GNS triple~\((\textup{L}^{2}(\G),\pi,\Lambda)\) 
   for~\(\psi\). There is a right \emph{multiplicative unitary} 
   \(\Multunit\in\U(\textup{L}^{2}(\G)\otimes\textup{L}^{2}(\G))\). Equivalently, 
   \(\Multunit\) satisfies the pentagon equation:
     \begin{equation}
      \label{eq:pentagon}
      \Multunit_{23}\Multunit_{12}
     = \Multunit_{12}\Multunit_{13}\Multunit_{23}
     \qquad
     \text{in \(\U(\textup{L}^{2}(\G)\otimes\textup{L}^{2}(\G)\otimes\textup{L}^{2}(\G)).\)}
   \end{equation}
   The right Haar weight version of the result~\cite{Kustermans-Vaes:LCQG}*{Proposition 
   6.10} ensures the manageability 
   (see~\cite{Soltan-Woronowicz:Remark_manageable}*{Definition 2.1}) 
   of~\(\Multunit\). The theory of manageable multiplicative unitaries 
   \cite{Woronowicz:Multiplicative_Unitaries_to_Quantum_grp} gives:
   \begin{enumerate}
   \item the dual multiplicative unitary 
   \(\DuMultunit\defeq\Flip\Multunit[*]\Flip\in\U(\textup{L}^{2}(\G)\otimes\textup{L}^{2}(\G))\) 
   is also manageable.
    \item the slices of~\(\Multunit\) defined by 
    \begin{align*}
     A &\defeq \{(\omega\otimes\Id_{\textup{L}^{2}(\G)})\Multunit :
    \omega\in\Bound(\textup{L}^{2}(\G))_*\}^\CLS ,\\
     \hat{A} &\defeq\{(\Id_{\textup{L}^{2}(\G)}\otimes\omega)\Multunit :
    \omega\in\Bound(\textup{L}^{2}(\G))_*\}^\CLS ,
   \end{align*} 
   are nondegenerate \(\Cst\)\nb-subalgebras of \(\Bound(\textup{L}^{2}(\G))\).
    \item  \(\Multunit\in\U(\hat{A}\otimes A)\subseteq\U(\textup{L}^{2}(\G)\otimes
  \textup{L}^{2}(\G))\). We write~\(\multunit\) for~\(\Multunit\) viewed as a unitary multiplier 
  of~\(\hat{A}\otimes A\);
    \item the comultiplication maps~\(\Comult[A]\) and 
  \(\DuComult[A]\) are characterised by the following conditions:  
  \begin{alignat}{2}
   \label{eq:aux_W_char_in_first_leg}
    (\Id_{A}\otimes\Comult)\multunit &=\multunit_{12}\multunit_{13}
   &\qquad\text{in~\(\U(A\otimes\hat{A}\otimes\hat{A})\)}\\
   \label{eq:W_char_in_first_leg}
  (\DuComult\otimes\Id_{A})\multunit &=\multunit_{23}\multunit_{13}
  &\qquad\text{in~\(\U(\hat{A}\otimes\hat{A}\otimes A)\).}
 \end{alignat}  
  \item there exist antiunitary involutive operators~\(J\) and~\(\hat{J}\)
  on~\(\textup{L}^{2}(\G)\). They implement the unitary 
  antipodes~\(\textup{R}\) and~\(\hat{R}\) on \(\G\) and~\(\DuG\) 
  as follows:
  \[
     \textup{R}(a)\defeq\hat{J}a^{*}\hat{J} 
     \quad\text{for~\(a\in A\)}
     \quad\text{and}\quad
     \hat{\textup{R}}(\hat{a})\defeq J a^{*}J
     \quad\text{for~\(\hat{a}\in\hat{A}\)}.
  \]   
\end{enumerate}
  The unitary \(\multunit\in\Mult(\hat{A}\otimes A)\) is called the \emph{reduced bicharacter} 
  of~\(\G\).
  \begin{definition}[\cite{Bedos-Tuset:Amen_coamen_lcqg}*{Definition 3.1}]
  A quantum group \(\Qgrp{G}{A}\) is \emph{coamenable} if it has a 
  \emph{bounded counit}. Equivalently, there is a unique 
  \Star{}homomorphism \(e^{A}\in\Mor(A,\C)\) such that  
 \begin{equation}
 \label{eq:coamen}
  (e^{A}\otimes\Id_{A})\Comult[A]=\Id_{A}.
 \end{equation}
\end{definition} 
 \section{Duality between quantum codoubles and Drinfeld doubles}
  \label{sec:drinf}
 Let~\(\Qgrp{G}{A}\) be a quantum group, 
 let \(\DuQgrp{G}{A}\) be its dual, and let~\(\multunit\in\U(\hat{A}\otimes 
 A)\) be the reduced bicharacter.
 
 Define 
 \[
    \CodoubAlg\defeq \hat{A}\otimes A
    \qquad\text{and}\qquad 
    \DuComult[\DrinfDoubAlg](\hat{a}\otimes a)\defeq 
    (\Id_{\hat{A}}\otimes\flip^{\multunit}\otimes\Id_{A})
    (\DuComult[A]\otimes\Comult[A]).
  \]  
  Here~\(\flip^{\multunit}\in\Mor(\hat{A}\otimes A,A\otimes\hat{A})\) 
  denotes the flip twisted by~\(\multunit\) defined by
  \(\flip^{\multunit}(\cdot)\defeq\flip(\multunit[*](\cdot)\multunit)\).
 Define~\(\Corep{U}\defeq\Multunit(\hat{J}\otimes J)\Multunit(\hat{J}\otimes J)\in\U
 (\textup{L}^{2}(\G)\otimes\textup{L}^{2}(\G))\). Theorem~\(8.7\) in 
 \cite{Masuda-Nakagami-Woronowicz:C_star_alg_qgrp}  
 shows that \(\Bialg{\CodoubAlg}\) is a quantum group and that 
 \begin{equation}
  \label{eq:mu_codoub}
   \DuMultunit[\DrinfDoubAlg]\defeq\Corep{U}_{12}\DuMultunit_{13}\Corep{U}^{*}_{12}\Multunit_{24}
   \in\U(\textup{L}^{2}(\G)\otimes\textup{L}^{2}(\G)\otimes\textup{L}^{2}(\G)
   \otimes\textup{L}^{2}(\G))
  \end{equation}
 is a manageable multiplicative unitary for it. The quantum codouble 
 of~\(\G\), denoted by \(\Codouble{\G}\),
 is the \(\Cst\)\nb-bialgebra \(\Bialg{\CodoubAlg}\) . 
\begin{definition}
 A pair of representations \(\rho\colon A\to\Bound(\Hils)\) 
 and~\(\theta\colon\hat{A}\to\Bound(\Hils)\) 
 is called a~\(\G\)\nb-\emph{Drinfeld pair} if it satisfies the 
 \(\G\)\nb-\emph{Drinfeld commutation relation}:
 \begin{equation}
  \label{eq:G-Drinfeld}
   \multunit_{1\rho}\multunit_{13}\multunit_{\theta3}
 = \multunit_{\theta3}\multunit_{13}\multunit_{1\rho}
 \qquad\text{in~\(\U(\hat{A}\otimes\Comp(\Hils)\otimes A)\).}
 \end{equation}
 Here~\(\multunit_{1\rho}\defeq ((\Id_{\hat{A}}\otimes\rho)\multunit)_{12}\) 
 and~\(\multunit_{\theta 3}\defeq ((\theta\otimes\Id_{A})\multunit)_{23}\).
 \end{definition}  
 Define~\(\rho\colon A\to\Bound(\textup{L}^{2}(\G)\otimes\textup{L}^{2}(\G))\) 
 and~\(\theta\colon\hat{A}\to\Bound(\textup{L}^{2}(\G)\otimes\textup{L}^{2}(\G))\) 
 by
 \[
   \rho(a)\defeq\Corep{U}(a\otimes 1_{\textup{L}^{2}(\G)})\Corep{U}^{*} 
   \qquad\text{and}\qquad
   \theta(\hat{a})\defeq 1_{\textup{L}^{2}(\G)}\otimes\hat{a}.
  \] 
  Here we drop the GNS representations of~\(A\) and 
  \(\hat{A}\) on~\(\textup{L}^{2}(\G)\). 
  \begin{proposition}
   \label{prop:Drinf_doub} 
   The pair \((\rho,\theta)\) is a \(\G\)\nb-Drinfeld pair.  
   Define \(\DrinfDoubAlg\defeq\rho(A)\cdot\theta(\hat{A})
   \subseteq\Bound(\textup{L}^{2}(\G)\otimes\textup{L}^{2}(\G))\) 
   and a map~\(\Comult[\DrinfDoubAlg]\colon\DrinfDoubAlg\to
   \Bound(\textup{L}^{2}(\G)\otimes\textup{L}^{2}(\G)\otimes
   \textup{L}^{2}(\G)\otimes\textup{L}^{2}(\G))\) by 
   \begin{equation}
    \label{eq:Drinf_doub_comult}
     \Comult[\DrinfDoubAlg](\rho(a)\cdot\theta(\hat{a}))
     \defeq(\rho\otimes\rho)\Comult[A](a)
     (\theta\otimes\theta)\DuComult[A](\hat{a})
     \quad\text{for~\(a\in A\), \(\hat{a}\in\hat{A}\).}
   \end{equation}
   Then~\(\Bialg{\DrinfDoubAlg}\) is the 
   dual quantum group of the quantum 
   codouble~\(\Codouble{\G}\).
  \end{proposition} 
  \begin{proof}
  The~\(\G\)\nb-Drinfeld commutation relation for 
  the pair~\((\rho,\theta)\) is equivalent to the following 
  relation
  \[
    \Corep{U}_{23}\Multunit_{12}\Corep{U}^{*}_{23}\Multunit_{14}
    \Multunit_{34}=\Multunit_{34}\Multunit_{14}
    \Corep{U}_{23}\Multunit_{12}\Corep{U}^{*}_{23};
  \]  
  then is one the intermediate steps in the proof of 
  Proposition~\(8.6\) in
  \cite{Masuda-Nakagami-Woronowicz:C_star_alg_qgrp}.
  
  The multiplicative unitary in~\eqref{eq:mu_codoub} is
  \(\DuMultunit[\DrinfDoubAlg]=\DuMultunit_{\rho 2}
  \Multunit_{\theta 3}\). The 
  manageability of~\(\Multunit\) and \(\DuMultunit\)  imply:
  \[
    \{(\Id_{\textup{L}^{2}(\G)\otimes\textup{L}^{2}(\G)}\otimes\omega\otimes
    \omega')\DuMultunit_{\rho 2}
  \Multunit_{\theta 3}:\omega,\omega'\in\Bound(\textup{L}^{2}(\G))_{*}
    \}^\CLS
    =\theta(\hat{A})\cdot\rho(A)=\DrinfDoubAlg.
  \]
  Since \(\DuMultunit[\DrinfDoubAlg]\) is also 
  manageable,~\(\DrinfDoubAlg\) is a nondegenerate~\(\Cst\)\nb-subalgebra 
  of~\(\Bound(\textup{L}^{2}(\G)\otimes\textup{L}^{2}(\G))\).
 
 Furthermore, we get 
 \(\DuMultunit_{\rho 2}\Multunit_{\theta 3}\in\U(\DrinfDoubAlg\otimes
 \CodoubAlg)\).
 
 The definition of~\(\Comult[\DrinfDoubAlg]\) gives
 \(\Comult[\DrinfDoubAlg]\in\Mor(\DrinfDoubAlg,\DrinfDoubAlg\otimes
 \DrinfDoubAlg)\). Hence it is sufficient to check that 
 \(\Comult[\DrinfDoubAlg]\) satisfies 
 \eqref{eq:W_char_in_first_leg} for 
 \(\multunit=\DuMultunit_{\rho 2}\Multunit_{\theta 3}\).
 We compute
  \begin{align*}
     (\Comult[\DrinfDoubAlg]\otimes\Id_{\hat{A}\otimes A})
        \DuMultunit_{\rho 2}\Multunit_{\theta 3}
    &=  \big(((\rho\otimes\rho)\DuComult\otimes\Id_{\hat{A}})\DuMultunit \big)_{123}
       \big(((\theta\otimes\theta)\Comult\otimes\Id_{A})\Multunit\big)_{124}\\
    &=  \DuMultunit_{\rho_{1}3}\DuMultunit_{\rho_{2}3}\Multunit_{\theta_{1}4}
          \Multunit_{\theta_{2}4}
      =  \DuMultunit_{\rho_{1}3}\Multunit_{\theta_{1}4}\DuMultunit_{\rho_{2}3}
          \Multunit_{\theta_{2}4}.    
  \end{align*} 
  The first equality 
  follows from~\eqref{eq:Drinf_doub_comult}; the second equality 
  uses~\eqref{eq:aux_W_char_in_first_leg} and 
  \eqref{eq:W_char_in_first_leg}, and~\(\rho_{i}\) or \(\theta_{i}\) 
  means~\(\rho\) or \(\theta\) acting on the~\(i\)\nb-th leg for~\(i=1,2\); 
  and the third equality uses the trivial commutation between 
  \(\DuMultunit_{\rho_{2}3}\) and~\(\Multunit_{\theta_{1}4}\).
 \end{proof} 
 \begin{definition}
  \label{def:Drinf_doub}
   The quantum group 
   \(\Drinfdouble{\G}\defeq\Bialg{\DrinfDoubAlg}\) is 
   the~\emph{Drinfeld double} of~\(\G\).
 \end{definition}  
  \begin{example}
  \label{ex:Drinf_doub_group_case} 
  Let~\(A=\Contvin(G)\) and~\(\hat{A}=\Cred(G)\) for  
  a locally compact group~\(G\). The underlying \(\Cst\)\nb-algebra 
  of the Drinfeld double of~\(G\) is \(\textup{C}(G))\rtimes_{\text{conj}}G\).
 \end{example} 
\section{The Haagerup property for the Drinfeld double}
 \label{sec:Haagerup_prop} 
  A (unitary) \emph{representation} of~\(\Qgrp{G}{A}\) on a Hilbert
  space~\(\Hils\) is a unitary \(\corep{V}\in\U(\Comp(\Hils)\otimes A)\)
  satisfying the following condition:
  \begin{equation}
    \label{eq:corep_cond}
    (\Id_{\Hils}\otimes\Comult[A])\corep{V} =\corep{V}_{12}\corep{V}_{13}
    \qquad\text{in }\U(\Comp(\Hils)\otimes A\otimes A).
  \end{equation}
  \begin{definition}[\cite{Daws-Fima-Skalski-
   White:Haagerup_prop_qnt_grp}*{Definition 5.1}] 
   \label{def:Haagerup_prop}
   A locally compact quantum group \(\Qgrp{G}{A}\) has the 
   \emph{Haagerup property} if there is a 
   representation~\(\corep{V}\in\U(\Comp(\Hils)\otimes A)\) 
   with the following properties:
   \begin{enumerate}
     \item there is a net~\(\{x_{i}\}\) of almost invariant 
              unit vectors in~\(\Hils\):   
              \begin{equation}
                \label{eq:almost_inv_vect_cond}
                \norm{\corep{V}(x_{i}\otimes \eta)-x_{i}\otimes\eta}
                \to 0 \quad\text{for all~\(\eta\in\textup{L}^{2}(\G)\).}
              \end{equation}
     \item\label{cond:mixing} \(\corep{V}\) is a mixing representation: 
              \((\omega_{x,y}\otimes\Id_{A})\corep{U}\in A\) for  all~\(x,y\in\Hils\),
              where \(\omega_{x,y}(T)\defeq \left\langle Tx, y\right\rangle\) is the vector functional.           
   \end{enumerate}
 \end{definition} 
  Next we prove our main result:
  \begin{theorem}
  \label{prop:Haagerup_prop_drinf_doub}
  Let~\(\G\) and~\(\DuG\) have the Haagerup property. 
  Then the Drinfeld double of~\(\G\) also has the Haagerup property.  
\end{theorem}   
\begin{proof} 
 Let~\(\corep{X}\in\U(\Comp(\Hils_{1})\otimes A)\) and 
 \(\corep{Y}\in\U(\Comp(\Hils_{2})\otimes\hat{A})\) be 
 representations of~\(\G\) and \(\DuG\), respectively.
 
  Define~\(\corep{V}\defeq \corep{X}_{1\rho}
 \corep{Y}_{2\theta}\in\U(\Comp(\Hils_{1}\otimes\Hils_{2})\otimes
 \DrinfDoubAlg_{\bichar})\). Equations~\eqref{eq:Drinf_doub_comult} and
 \eqref{eq:corep_cond} for~\(\corep{X}\) and~\(\corep{Y}\) give:
 \[
       (\Id_{\Hils_{1}\otimes\Hils_{2}}\otimes\Comult[\DrinfDoubAlg_{\bichar}])
       \corep{X}_{1\rho}\corep{Y}_{2\theta}
    = \corep{X}_{1\rho_{3}}\corep{X}_{1\rho_{4}}\corep{Y}_{2\theta_{3}}
       \corep{Y}_{2\theta_{4}}
    =  \corep{X}_{1\rho_{3}}\corep{Y}_{2\theta_{3}}
        \corep{X}_{1\rho_{4}}\corep{Y}_{2\theta_{4}}.
\]    
 Hence \(\corep{V}\) is a representation of the Drinfeld double 
 \(\Drinfdouble{\G}\) on~\(\Hils_{1}\otimes\Hils_{2}\). 
 
 Let~\(\corep{X}\) and~\(\corep{Y}\) satisfy the 
 mixing condition~\ref{cond:mixing} in  
 Definition~\ref{def:Haagerup_prop}. Then 
 we get the mixing condition for~\(\corep{V}\):
 \[
    (\omega_{x_{1}\otimes y_{1},x_{2}\otimes y_{2}}\otimes\Id_{\DrinfDoubAlg_{\bichar}})
    (\corep{X}_{1\rho}\corep{Y}_{2\theta})
    =\rho\big((\omega_{x_{1},y_{1}}\otimes\Id_{A})\corep{X}\big)
       \theta\big(\omega_{y_{1},y_{2}}\otimes\Id_{\hat{A}})\corep{Y}\big)
    \in\DrinfDoubAlg_{\bichar}
\]
 for all~\(x_{1}, x_{2}\in\Hils_{1}\), \(y_{1}, y_{2}\in\Hils_{2}\).

 Let~\(\Corep{X}\in\Hils[L](\Hils_{1}\otimes A)\) be the image of 
 \(\corep{X}\in\U(\Comp(\Hils_{1})\otimes A)\) under the 
 canonical isomorphism between \(\Mult(\Comp(\Hils_{1})\otimes A)\)
 and~\(\Hils[L](\Hils_{1}\otimes A)\) (the space of adjointable operators on the 
 Hilbert~\(A\)\nb-module~\(\Hils_{1}\otimes A\)). By 
 \cite{Daws-Fima-Skalski-White:Haagerup_prop_qnt_grp}*{Proposition 2.7},  
 condition~\eqref{eq:almost_inv_vect_cond} becomes equivalent to the 
 following condition:
  \begin{equation}
    \label{eq:eqiv_almost_inv_vect_cond}
      \norm{\Corep{X}(x_{i}\otimes a)-x_{i}\otimes a}
       \to 0 \quad\text{for all~\(a\in A\).}
   \end{equation} 
 Let~\(\{x_{i}\}\) be the net of almost invariant 
 unit vectors in~\(\Hils_{1}\) for~\(\corep{X}\). 
 Let \(d\in\DrinfDoubAlg_{\bichar}\) and let 
 \(\{e^{A}_{\lambda}\}\) be a  bounded approximate identity in~\(A\). 
 The construction of~\(\DrinfDoubAlg\) gives \(\rho\in\Mor(A,\DrinfDoubAlg)\).
 Given any~\(\epsilon>0\) there exists~\(\lambda'\) such that
 \(\norm{\rho(e^{A}_{\lambda'})d-d}<\epsilon/3\); hence 
 \[
   \norm{\Corep{X}_{1\rho}(x_{i}\otimes d)-\Corep{X}_{1\rho}(x_{i}\otimes\rho(e^{A}
   _{\lambda'})d)}<\epsilon/3,
  \qquad
   \norm{x_{i}\otimes\rho(e^{A}_{\lambda'})d-x_{i}\otimes d}<\epsilon/3
 \]  
 for each~\(x_{i}\).
 Furthermore, by~\eqref{eq:eqiv_almost_inv_vect_cond}, 
 there exists \(i'\) such that  for~\(i\geq i'\)
 \[
   \norm{\Corep{X}_{1\rho}(x_{i}\otimes\rho(e^{A}_{\lambda'})d)
   -x_{i}\otimes\rho(e^{A}_{\lambda'})d}\leq\norm{d}\cdot
   \norm{(\Corep{X}_{1\rho}(x_{i}\otimes\rho(e^{A}_{\lambda'}))
   -x_{i}\otimes\rho(e^{A}_{\lambda'})}<\epsilon/3.
 \]
 The last two estimates together give
 \begin{multline*}
    \norm{\Corep{X}_{1\rho}(x_{i}\otimes d)-(x_{i}\otimes d)}
    \leq 
       \norm{\Corep{X}_{1\rho}(x_{i}\otimes d)-\Corep{X}_{1\rho}
       (x_{i}\otimes\rho(e^{A}_{\lambda'})d)}+\\
    \norm{\Corep{X}_{1\rho}(x_{i}\otimes\rho(e^{A}_{\lambda'})d)
    -x_{i}\otimes\rho(e^{A}_{\lambda'})d}+
    \norm{x_{i}\otimes\rho(e^{A}_{\lambda'})d-x_{i}\otimes d}
    <\epsilon\quad\text{for~\(i\geq i'\).}
 \end{multline*}  
 Since~\(\epsilon>0\) is arbitrary, we get 
 \(\lim\limits_{i}\norm{\Corep{X}_{1\rho}(x_{i}\otimes d)-(x_{i}\otimes 
 d)}=0\) for~\(d\in\DrinfDoubAlg_{\bichar}\).
  
 Similarly, for a net of almost invariant 
 unit vectors~\(\{y_{j}\}\) in~\(\Hils_{2}\) for~\(\corep{Y}\),  
 we have 
 \(
   \lim\limits_{j}\norm{\Corep{Y}_{1\theta}(y_{j}\otimes d)
   -y_{j}\otimes d}=0\) for~\(d\in\DrinfDoubAlg_{\bichar}\).
   
 Finally, we show that~\(\{x_{i}\otimes y_{j}\}\) is an almost invariant unit vector for 
 \(\Corep{X}_{1\rho}\Corep{Y}_{2\theta}\). We compute 
 \begin{align*}
  & \norm{\Corep{X}_{1\rho}\Corep{Y}_{2\theta}(x_{i}\otimes y_{j}\otimes d)
  -x_{i}\otimes y_{j}\otimes d}\\
 &\leq \norm{\Corep{X}_{1\rho}\Corep{Y}_{2\theta}(x_{i}\otimes y_{j}\otimes d)
   -\Corep{X}_{1\rho}(x_{i}\otimes y_{j}\otimes d)} +
   \norm{\Corep{X}_{1\rho}(x_{i}\otimes y_{j}\otimes d)-x_{i}\otimes y_{j}\otimes d}.
\end{align*}     
Therefore,~\(\lim\limits_{i,j}\norm{\Corep{X}_{1\rho}\Corep{Y}_{2\theta}
(x_{i}\otimes y_{j}\otimes d)-x_{i}\otimes y_{j}\otimes d}=0\) for~\(d\in\DrinfDoubAlg_{\bichar}\). 
\end{proof}

   \begin{corollary}
    \label{cor:Haag_drinf}
      Assume~\(\G\) and~\(\DuG\) are coamenable. Then the 
      Drinfeld double~\(\Drinfdouble{\G}\) has the Haagerup property. 
   \end{corollary}
   \begin{proof} 
    By~\cite{Daws-Fima-Skalski-White:Haagerup_prop_qnt_grp}*{Proposition 5.2}, both 
    \(\G\) and~\(\DuG\) have the Haagerup property; hence \(\Drinfdouble{\G}\) 
    has the Haagerup property. 
   \end{proof}
\begin{remark}
 The proof of Theorem~\ref{prop:Haagerup_prop_drinf_doub} also works in more general 
 framework of locally compact quantum groups constructed from manageable multiplicative 
 unitaries~\cite{Woronowicz:Mult_unit_to_Qgrp}. Furthermore, one can generalise 
 the statement of Theorem~\ref{prop:Haagerup_prop_drinf_doub} 
 by replacing~\(\DuG\) by another locally compact quantum group 
 \(\G[H]\) with the Haagerup property, and the Drinfeld of~\(\G\) by generalised 
 Drinfeld double of \(\G\) and~\(\G[H]\) with respect to a given bicharacter 
 (see~\cite{Roy:Codoubles}) between them. 
\end{remark}
\begin{proposition}
 \label{prop:Haagerup_1} 
 The Drinfeld doubles of locally compact groups~\(G\) with the Haagerup property, 
 coamenable compact quantum groups~\(\G\), extended quantum \(\textup{SU}(1,1)\), 
 quantum~\(ax+b\), quantum~\(az+b\) and quantum \(\textup{E}(2)\) groups have
 the Haagerup property, respectively. In particular, the Drinfeld doubles of~\(\mathbb{F}_{n}\) 
 (free group with~\(n(\geq 2)\) generators) and the duals of quantum Lorentz groups have the Haagerup property.
\end{proposition}
\begin{proof}
 Let~\(G\) be a locally compact group; hence~\(\Contvin(G)\) 
 is coamenable. By~\cite{Daws-Fima-Skalski-White:Haagerup_prop_qnt_grp}*{Proposition 5.2}, 
 \(\Cred(G)\) (as a quantum group) has the Haagerup property. Furthermore, if~\(G\) has 
 the Haagerup property if and only if~\(\Contvin(G)\) has the Haagerup property.
 Consequently, by Theorem~\ref{prop:Haagerup_prop_drinf_doub}, 
 the Drinfeld double of~\(G\), namely, \(\Contvin(G)\rtimes_{\textup{conj}}G\) 
 has the Haagerup property if and only if~\(G\) has the Haagerup 
 property.  In particular, the Drinfeld doubles of~\(\mathbb{F}_{n}\) 
 (free group with~\(n(\geq 2)\) generators) have the Haagerup property.  
 
 It is shown in~\cite{Caspers:Weak_amenb_SU_1_1} that  
 the extended quantum~\(\textup{SU}(1,1)\) groups  
 have the Haagerup property. By \cite{Caspers:Weak_amenb_SU_1_1}*{Theorem 8.3}, 
 their duals  also have the Haagerup property. By Theorem~\ref{prop:Haagerup_prop_drinf_doub}, 
 Drinfeld doubles of the extended~\(\textup{SU}_{q}(1,1)\)  has the 
 Haagerup property.
 
  Let~\(\G\) be a compact quantum group. By,  
 \cite{Bedos-Murphy-Tuset:Amen_coamen_alg_qgrp}*{Proposition 4.1}, 
 the discrete quantum group~\(\DuG\) is always coamenable. Additionally,  
 if \(\G\) is coamenable, by Corollary~\ref{cor:Haag_drinf}, 
 \(\Drinfdouble{\G}\) is amenable and has the Haagerup property. 
 In particular, \(\textup{SU}_{q}(2)\) is a compact quantum group and 
 coamenable (see~\cite{Bedos-Tuset-Murphy:Coamen_cpt_qgrp}). Hence, 
 the Drinfeld double of~\(\textup{SU}_{q}(2)\), dual of the quantum 
 Lorentz group in~\cite{Podles-Woronowicz:Quantum_deform_Lorentz}, 
 has the Haagerup property.
 
 Finally, \cite{Daws-Fima-Skalski-White:Haagerup_prop_qnt_grp}*{Example 5.4} 
 shows that quantum~\(ax+b\), quantum~\(az+b\) and quantum \(\textup{E}(2)\) 
 groups have the Haagerup property. Therefore, by Corollary~\ref{cor:Haag_drinf} 
 we conclude that their respective Drinfeld doubles have the Haagerup property. 
\end{proof}

\section*{Acknowledgements}
This research was conducted at the Fields Institute during
the thematic program on Abstract Harmonic Analysis, Banach and 
Operator Algebras while the author was financially supported by a Fields\nb-Ontario postdoctoral 
fellowship. Apart from that, parts of the manuscript had been revised while the author had been 
additionally supported by NSERC and ERA at the University of Ottawa. 
The author gratefully thanks Professor Nico Spronk for his helpful discussions and comments.


\begin{bibdiv}
  \begin{biblist}
\bib{Baaj-Skandalis:Unitaires}{article}{
  author={Baaj, Saad},
  author={Skandalis, Georges},
  title={Unitaires multiplicatifs et dualit\'e pour les produits crois\'es de $C^*$\nobreakdash -alg\`ebres},
  journal={Ann. Sci. \'Ecole Norm. Sup. (4)},
  volume={26},
  date={1993},
  number={4},
  pages={425--488},
  issn={0012-9593},
  review={\MRref {1235438}{94e:46127}},
  eprint={http://www.numdam.org/item?id=ASENS_1993_4_26_4_425_0},
}

\bib{Baaj-Vaes:Double_cros_prod}{article}{
  author={Baaj, Saad},
  author={Vaes, Stefaan},
  title={Double crossed products of locally compact quantum groups},
  journal={J. Inst. Math. Jussieu},
  volume={4},
  date={2005},
  number={1},
  pages={135--173},
  issn={1474-7480},
  review={\MRref {2115071}{2006h:46071}},
  doi={10.1017/S1474748005000034},
}

\bib{Bedos-Murphy-Tuset:Amen_coamen_alg_qgrp}{article}{
  author={B\'edos, Erik},
  author={Murphy, Gerard J.},
  author={Tuset, Lars},
  title={Amenability and coamenability of algebraic quantum groups},
  journal={Int. J. Math. Math. Sci.},
  volume={31},
  date={2002},
  number={10},
  pages={577--601},
  issn={0161-1712},
  review={\MRref {1931751}{2003j:46107}},
  doi={10.1155/S016117120210603X},
}

\bib{Bedos-Tuset:Amen_coamen_lcqg}{article}{
  author={B\'edos, Erik},
  author={Tuset, Lars},
  title={Amenability and co-amenability for locally compact quantum groups},
  journal={Internat. J. Math.},
  volume={14},
  date={2003},
  number={8},
  pages={865--884},
  issn={0129-167X},
  review={\MRref {2013149}{2004k:4612}},
  doi={10.1142/S0129167X03002046},
}

\bib{Bedos-Tuset-Murphy:Coamen_cpt_qgrp}{article}{
  author={B\'edos, Erik},
  author={Murphy, Gerard. J.},
  author={Tuset, Lars},
  title={Co-amenability of compact quantum groups},
  journal={J. Geom. Phys.},
  volume={40},
  date={2001},
  number={2},
  pages={130--153},
  issn={0393-0440},
  review={\MRref {1862084}{2002m:46100}},
  doi={10.1016/S0393-0440(01)00024-9},
}

\bib{Caspers:Weak_amenb_SU_1_1}{article}{
  author={Caspers, Martijn},
  title={Weak amenability of locally compact quantum groups and approximation properties of extended quantum $\textup {SU}(1,1)$},
  journal={Comm. Math. Phys.},
  status={to appear},
  date={2013},
  note={\arxiv {1306.4558v2}},
}

\bib{Daws-Fima-Skalski-White:Haagerup_prop_qnt_grp}{article}{
  author={Daws, Matthew},
  author={Fima, Pierre},
  author={Skalski, Adam G.},
  author={White, Stuart A.},
  title={The Haagerup property for locally compact quantum groups},
  journal={J. Reine Angew. Math.},
  status={to appear},
  date={2014},
  note={\arxiv {1303.3261v3}},
}

\bib{Drinfeld:Quantum_groups}{article}{
  author={Drinfel'd, V. G.},
  title={Quantum groups},
  conference={ title={International Congress of Mathematicians}, place={Berkeley, Calif.}, date={1986}, },
  pages={798--820},
  publisher={Amer. Math. Soc.},
  place={Providence, RI},
  date={1987},
  review={\MRref {934283}{89f:17017}},
}

\bib{Koelink-Kustermans:normalizer_SU_1_1}{article}{
  author={Koelink, Erik},
  author={Kustermans, Johan},
  title={A locally compact quantum group analogue of the normalizer of $\mathrm {SU}(1,1)$ in $\mathrm {SL}(2,\mathbb C)$},
  journal={Comm. Math. Phys.},
  volume={233},
  date={2003},
  number={2},
  pages={231--296},
  issn={0010-3616},
  review={\MRref {1962042}{2004c:46136}},
}

\bib{Kustermans-Vaes:LCQG}{article}{
  author={Kustermans, Johan},
  author={Vaes, Stefaan},
  title={Locally compact quantum groups},
  journal={Ann. Sci. \'Ecole Norm. Sup. (4)},
  volume={33},
  date={2000},
  number={6},
  pages={837--934},
  issn={0012-9593},
  review={\MRref {1832993}{2002f:46108}},
  doi={10.1016/S0012-9593(00)01055-7},
}

\bib{Masuda-Nakagami-Woronowicz:C_star_alg_qgrp}{article}{
  author={Masuda, Tetsuya},
  author={Nakagami, Y.},
  author={Woronowicz, Stanis\l aw Lech},
  title={A $C^*$\nobreakdash -algebraic framework for quantum groups},
  journal={Internat. J. Math},
  volume={14},
  date={2003},
  number={9},
  pages={903--1001},
  issn={0129-167X},
  review={\MRref {2020804}{2004j:46100}},
  doi={10.1142/S0129167X03002071},
}

\bib{Nest-Voigt:Poincare}{article}{
  author={Nest, Ryszard},
  author={Voigt, {Ch}ristian},
  title={Equivariant Poincar\'e duality for quantum group actions},
  journal={J. Funct. Anal.},
  volume={258},
  date={2010},
  number={5},
  pages={1466--1503},
  issn={0022-1236},
  review={\MRref {2566309}{2011d:46143}},
  doi={10.1016/j.jfa.2009.10.015},
}

\bib{Podles-Woronowicz:Quantum_deform_Lorentz}{article}{
  author={Podle\'s, Piotr},
  author={Woronowicz, Stanis\l aw Lech},
  title={Quantum deformation of Lorentz group},
  journal={Comm. Math. Phys.},
  volume={130},
  date={1990},
  number={2},
  pages={381--431},
  issn={0010-3616},
  review={\MRref {1059324}{91f:46100}},
  eprint={http://projecteuclid.org/euclid.cmp/1104200517},
}

\bib{Roy:Codoubles}{article}{
  author={Roy, Sutanu},
  title={The Drinfeld double for $C^*$\nobreakdash -algebraic quantum groups},
  journal={J. Operator Theory},
  status={accepted},
  note={\arxiv {1404.5384v4}},
  date={2015},
}

\bib{Soltan-Woronowicz:Remark_manageable}{article}{
  author={So\l tan, Piotr M.},
  author={Woronowicz, Stanis\l aw Lech},
  title={A remark on manageable multiplicative unitaries},
  journal={Lett. Math. Phys.},
  volume={57},
  date={2001},
  number={3},
  pages={239--252},
  issn={0377-9017},
  review={\MRref {1862455}{2002i:46072}},
  doi={10.1023/A:1012230629865},
}

\bib{Woronowicz:Qnt_E2_and_Pontr_dual}{article}{
  author={Woronowicz, Stanis\l aw Lech},
  title={Quantum $E(2)$ group and its Pontryagin dual},
  journal={Lett. Math. Phys.},
  volume={23},
  date={1991},
  number={4},
  pages={251--263},
  issn={0377-9017},
  review={\MRref {1152695}{93b:17058}},
  doi={10.1007/BF00398822},
}

\bib{Woronowicz:Mult_unit_to_Qgrp}{article}{
  author={Woronowicz, Stanis\l aw Lech},
  title={From multiplicative unitaries to quantum groups},
  journal={Internat. J. Math.},
  volume={7},
  date={1996},
  number={1},
  pages={127--149},
  issn={0129-167X},
  review={\MRref {1369908}{96k:46136}},
  doi={10.1142/S0129167X96000086},
}

\bib{Woronowicz:CQG}{article}{
  author={Woronowicz, Stanis\l aw Lech},
  title={Compact quantum groups},
  conference={ title={Sym\'etries quantiques}, address={Les Houches}, date={1995}, },
  book={ publisher={North-Holland}, place={Amsterdam}, },
  date={1998},
  pages={845--884},
  review={\MRref {1616348}{99m:46164}},
}

\bib{Woronowicz:Multiplicative_Unitaries_to_Quantum_grp}{article}{
  author={Woronowicz, Stanis\l aw Lech},
  title={From multiplicative unitaries to quantum groups},
  journal={Internat. J. Math.},
  volume={7},
  date={1996},
  number={1},
  pages={127--149},
  issn={0129-167X},
  review={\MRref {1369908}{96k:46136}},
  doi={10.1142/S0129167X96000086},
}

\bib{Woronowicz:Quantum_azb}{article}{
  author={Woronowicz, Stanis\l aw Lech},
  title={Quantum `$az+b$' group on complex plane},
  journal={Internat. J. Math.},
  volume={12},
  date={2001},
  number={4},
  pages={461--503},
  issn={0129-167X},
  review={\MRref {1841400}{2002g:46120}},
  doi={10.1142/S0129167X01000836},
}

\bib{Woronowicz-Zakrzewski:Quantum_axb}{article}{
  author={Woronowicz, Stanis\l aw Lech},
  author={Zakrzewski, Stanis\l aw},
  title={Quantum `$ax+b$' group},
  journal={Rev. Math. Phys.},
  volume={14},
  date={2002},
  number={7-8},
  pages={797--828},
  issn={0129-055X},
  review={\MRref {1932667}{2003h:46106}},
  doi={10.1142/S0129055X02001405},
}
\end{biblist}
\end{bibdiv}
\end{document}